\documentclass[reqno,11pt]{amsart}
\usepackage{amsmath}
\usepackage{amssymb}
\usepackage{latexsym}
\usepackage{amsthm}
\usepackage{fullpage}
\usepackage[colorinlistoftodos]{todonotes}
\usepackage{marginnote}
\usepackage{enumitem}
\usepackage{mathtools}
\usepackage{dsfont}
\usepackage{mdframed}
\usepackage{xcolor}
\usepackage{hyperref}
\usepackage{esint}
\usepackage{amsfonts,latexsym,amsmath,amscd,geometry}
\numberwithin{equation}{section}

\newtheorem{theorem}{Theorem}[section]
\newtheorem{lemma}[theorem]{Lemma}

\newtheorem{remark}[theorem]{Remark}

\definecolor{vine}{rgb}{0.7,0.1,0.1}
\definecolor{vine}{rgb}{0.7,0.1,0.1}
\definecolor{darkgreen}{rgb}{0,0.5,0}

\newcommand{\R}{{\mathbb R}}

\renewcommand{\div}{\operatorname{div}\,}
\newcommand{\curl}{{{\operatorname{curl}}\,}}

\newcommand \nc{\newcommand}

\nc{\ba}{\begin{array}}\nc{\ea}{\end{array}}
\nc{\be}{\begin{eqnarray}}\nc{\ee}{\end{eqnarray}}
\nc{\beq}{\begin{equation}}\nc{\eeq}{\end{equation}}
\nc{\bex}{\begin{eqnarray*}}\nc{\eex}{\end{eqnarray*}}
\nc{\btm}{\begin{theorem}} \nc{\etm}{\end{theorem}}
\nc{\blm}{\begin{lemma}} \nc{\elm}{\end{lemma}}

\nc{\ld}{\lambda}
\nc{\va}{\varphi}
\nc{\ve}{\varepsilon}

\def\curl{\mbox{curl\,}}

\begin{document}
\author{Jeaheang Bang}
\address{Institute of Theoretical Sciences, Westlake University, Hangzhou, Zhejiang, 310030, P. R. China}
\email{jhbang@westlake.edu.cn}

\author{Changyou Wang}
\address{Department of Mathematics, Purdue University,
 150 N. University Street,
 West Lafayette, IN 47907-2067, USA}
\email{wang2482@purdue.edu}
\title{On Rigidity of the Steady Ericksen-Leslie System}

\begin{abstract}
    We study solutions, with scaling-invariant bounds, to the steady simplified Ericksen-Leslie system in $\mathbb{R}^n\setminus \{0\}$.
    When $n=2$, we construct and classify a class of self-similar solutions.  When $n\ge 3$, 
    we establish the rigidity asserting that if $(u,d)$ satisfies a scaling-invariant bound with a small constant, then 
    $u\equiv 0$ and $d=$ constant for $n\geq 4$ or  $u$ is a Landau solution and 
    $d=$ constant for $n=3$. Such a smallness condition can be weaken when $n=4$ or the solutions are self-similar.
\end{abstract}
\maketitle
\section{Introduction}

The Ericksen-Leslie system, proposed by Ericksen \cite{Ericksen1962} and Leslie \cite{Leslie1968}, models the hydrodynamic motion of nematic liquid crystal materials under the influence of both the underlying fluid velocity field and the macroscopic average orientation field of the liquid crystal molecules.  A simplified Ericksen-Leslie system proposed by Lin \cite{Lin1989} reads as follows. For $n=2,3$ and $\Omega\subset\mathbb{R}^n$,
let $(u, d):\Omega\times (0,\infty)\to \mathbb R^n\times \mathbb{S}^2$ solve
 \begin{align}
    \label{EL}
    \left\{
    \begin{aligned}
        \partial_t u-\Delta u + u\cdot \nabla u + \nabla p &=
        -\div{(\nabla d \odot \nabla d)},
        \
        \div u=0,
        \\
        \partial_td+ u\cdot \nabla d
        &=\Delta d+ |\nabla d|^2 d.
    \end{aligned}
    \right.
    \end{align}
Here $u$ and $d$ represent the underlying fluid velocity field and the averaged 
orientation field of the nematic liquid crystal molecules respectively, and $p$ stands for the pressure function. Here $\nabla d\odot\nabla d$ represents the Ericksen stress tensor defined by
$
(\nabla d \odot \nabla d)_{ij}=\partial _{x_i} d \cdot \partial_{x_j} d.
$
Mathematically, the Ericksen-Leslie system \eqref{EL} is a system strongly coupling the forced Navier-Stokes system with the transported heat flow of harmonic maps into the unit sphere $\mathbb{S}^2\subset \mathbb{R}^3$.

Due to both the supercritical nonlinearities and their strong couplings, it has been a challenging problem to analyze \eqref{EL}, such as the existence, uniqueness, and regularity of solutions. In dimensions $n=2$, Lin-Lin-Wang \cite{LLW2010} 
established the global existence of a weak solution 
of the initial-boundary value problem of \eqref{EL}, that is smooth away from at most finitely many singular times (see also \cite{Hong2011, HongXin2012} for related results for $\Omega=\mathbb{R}^2$), while the uniqueness was proven by Lin-Wang \cite{LW2011} (see also \cite{XuZhang2012} and \cite{LTX2016}). When the dimension $n=3$, there are fewer results available in the literature. Wang \cite{Wang2011} obtained the well-posedness for \eqref{EL} on $\mathbb{R}^3$, when an initial data $(u_0, d_0)$ has small ${\rm{BMO}}^{-1}\times{\rm{BMO}}$-norm, and Lin-Wang \cite{LinWang2016} obtained the global existence of Leray-Hopf solutions of \eqref{EL} when the initial data $d_0(\mathbb{R}^3)\subset\mathbb{S}^2_+$. The existence of finite time singularities of \eqref{EL} 
on $\mathbb{R}^n$ has been constructed by Huang-Lin-Liu-Wang \cite{HLLW2016} for $n=3$ and Lai-Lin-Wang-Wei-Zhou \cite{LLWWZ2022} for $n=2$. However, those singularities $(x_*, T_*)$ are of the Type II singularities, i.e.,
\begin{align}
\label{t2}
\limsup_{t\uparrow T_*}\sqrt{T_*-t}\big\|(u(t),\nabla d(t))\big\|_{C^0(\mathbb{R}^n)}=\infty.    
\end{align}
This naturally leads to the question that
if \eqref{EL} admits finite time Type I singularity, 
namely, there exists a constant $C_*>0$ such that
\begin{align}
\label{t1}
\big\|(u(t),\nabla d(t))\big\|_{C^0(\mathbb{R}^n)}
\le \frac{C_*}{\sqrt{T_*-t}}, \ 0<t<T_*.    
\end{align}
It is well-known that the existence of Type I singularity of \eqref{EL} is closely related to the existence of nontrivial self-similar solutions $(u_*,d_*)$ of
\eqref{EL} in $\mathbb{R}^n\times (-\infty, 0)$:
$$
(\lambda u_*, d_*)(\lambda x,\lambda^2 t)=
(u_*, d_*)(x,t), \ \forall\lambda>0,\ (x,t)\in 
\mathbb{R}^n\times (-\infty, 0).
$$
This is an open question that we plan to exploit in the near future. In this paper, we will focus on solutions of the steady form of \eqref{EL}, which can also be viewed as the limiting equation of \eqref{EL} at the time infinity. More specifically, we will study the Liouville property of certain solutions to the steady Ericksen-Leslie system in $\mathbb{R}^n$:
    \begin{align}
    \label{SEL}
    \left\{
    \begin{aligned}
        -\Delta u + u\cdot \nabla u + \nabla p &=
        -\div{(\nabla d \odot \nabla d)},
        \ \div u=0,
        \\
        \Delta d+|\nabla d|^2 d&= u\cdot \nabla d.
    \end{aligned}
    \right.
    \end{align}
The system \eqref{SEL} is strongly coupling between the steady Navier-Stokes equation and the harmonic map equation. 

Note that the system \eqref{SEL} enjoys the scaling property, i.e., if $(u,p,d)$ is a solution to \eqref{SEL}, then 
$\left( \lambda \, u(\lambda x), \lambda^2 \, p (\lambda x), d(\lambda x) \right)$ is also a solution of
\eqref{SEL}, for any $\lambda >0$. Motivated by this, we will also study\emph{self-similar solutions} $(u,p,d)$ of \eqref{SEL}, that is,
    \begin{align*}
        u(x)=\lambda  u(\lambda x),
        \quad p(x)= \lambda^2  p(\lambda x),
        \quad d(x)= d(\lambda x) 
        \quad \text{in }\ \mathbb{R}^n \setminus \{0\}, \ \forall \lambda>0.
    \end{align*}

We would also like to remark that in general, self-similar solutions of \eqref{SEL} are useful to study the asymptotic behavior of general solutions of \eqref{EL}  near $|x|=\infty$ or an isolated singularity of \eqref{SEL}.

When $d$ 
is a constant, the equation \eqref{SEL} reduces to the steady Navier-Stokes equation:
     \begin{align}
    \label{SNS}
    \begin{aligned}
        -\Delta u + u\cdot \nabla u + \nabla p =0,
        \ \div u=0.
    \end{aligned}
    \end{align}
There have been a number of studies on self-similar solutions to \eqref{SNS}. For instance, 
when $n=2$, all nontrivial self-similar solutions $u\in C^\infty (\mathbb{R}^2\setminus \{0\})$ to \eqref{SNS} have been found by \cite{Hamel17}, \cite{Sverak11}, \cite{GuillodWittwer15SIAM},
referred as Hamel solutions.

When $n=3$, all self-similar solutions to \eqref{SNS} have been found and characterized  in \cite{Landau44}, \cite{Sverak11}, \cite{TianXin98}. Such solutions are called Landau solutions, which are given by
    \begin{align*}
	u_a(x)
	=
	\curl (\psi_a e_\theta),
    \quad \psi_a=
    \frac{2\sin \phi}{a-\cos \phi},
	\end{align*}
where $\phi\in (0,\pi)$ is the polar angle and $\theta\in [0,2\pi)$ the azimuthal angle in spherical coordinates, and the parameter $a\in (1,\infty]$ (see also \cite{Tsai18}).
However, when $n\geq 4$, it was proved by\ \cite{BangGuiLiuWangXie23} that any solution $u\in C^\infty(\mathbb{R}^n\setminus \{0\})$ to \eqref{SNS} must vanish, if it satisfies
    \begin{align*}
        |u(x)|\leq \frac{C}{|x|} \quad \text{in }\mathbb{R}^n\setminus \{0\}.
    \end{align*}

When $u\equiv 0$ and $p$ is constant, the equation \eqref{SEL} reduces to the equation of harmonic maps:
    \begin{align}
    \label{HM}
        \Delta d +|\nabla d|^2 d=0.
    \end{align}
When $n=2$ and $d:\mathbb{R}^2 \to \mathbb{S}^1\subset\mathbb{S}^2$, 
all nontrivial self-similar solutions of \eqref{HM}, in polar coordinates, are given by
    \begin{align*}
        d(r,\theta)= \cos (m\theta+\theta_0) e_r + \sin (m\theta+\theta_0) e_\theta,
    \end{align*}
for an integer $m$ and a constant $\theta_0\in (0,2\pi)$, where $\{e_r,e_\theta\}$ denotes the basis vectors of the polar coordinate system.
When $n=3$, $d(x)=\phi(\frac{x}{|x|})$, where
$\phi:\mathbb{S}^2\cong\overline{\mathbb{C}^2}
\to\mathbb{S}^2\cong\overline{\mathbb{C}^2}$ is a rational function in $z$ or $\bar{z}$ (see \cite{BCL1986}).

Here, we prove the following result for self-similar solutions of \eqref{SEL} for $n=2$.
\begin{theorem}
\label{2Dthm} 

i) For any $c, \theta_0\in\mathbb{R}, m\in\mathbb{Z}$, 
the triple $(u_c, p_{c, m}, 
d_m)$ given by
$$
u_c=\frac{c}{r} e_r,\ \
p_{c,m}=\frac{(m+1)^2-c^2}{2r^2}, \  \ d_m=\cos(m\theta+\theta_0)e_r
+\sin(m\theta+\theta_0)e_\theta
$$
is a self-similar solution of
\eqref{SEL} on $\mathbb{R}^2\setminus\{0\}$.\\
ii) For any $\Phi\in \mathbb{R}, k\in \mathbb{N}$, with $
        4+\frac{\Phi}{\pi}\leq k^2,$
there exists a nontrivial periodic function $f_{\Phi,k}$, with the minimal period $\frac{2\pi}{k}$, satisfying 
$
\int_0^{2\pi} f_{\Phi,k} (\theta) \, d\theta = \Phi,
$
such that, for $\theta_1, \theta_2\in \mathbb{R}, m\in \mathbb{Z}$, the triple $(u_{\Phi,k}, p_{\Phi,k,m}, d_m)$ given by
    \begin{align*}
        u_{\Phi,k}
        &=\frac{f_{\Phi,k}(\theta+\theta_1) }{r}e_r,
        \\
        p_{\Phi,k,m}
        &=
        \frac{1}{r^2}
        \big( 
        2f_{\Phi,k} (\theta+\theta_1)
        -\frac{1}{4\pi} \int_0^{2\pi} |f_{\Phi,k}(\theta)|^2 \, d\theta
        -\frac{2\Phi}{\pi}
        +\frac{1}{2}(m+1)^2
        \big),
        \\
        d_m
        &= \cos (m\theta+\theta_2) e_r + \sin (m\theta+ \theta_2) e_\theta
    \end{align*}
    is a self-similar solution of
\eqref{SEL} on $\mathbb{R}^2\setminus\{0\}$.\\
iii) For any $ \mu\not=0, \Psi, \theta_3 \in \mathbb{R}$,
the triple $(u_{\Psi,\mu}, p_{\Psi,\mu}, d_{-1})$ given by
    \begin{align*}
        u_{\Psi,\mu}
        =
        \frac{\Psi}{2\pi r} e_r+ \frac{\mu}{r}e_\theta, 
        \quad 
        p_{\Psi,\mu}
        =-\frac{|u_{\Psi,\mu}|^2}{2},
        \quad 
        d_{-1}= \cos (-\theta+\theta_3) e_r+ \sin (-\theta+\theta_3) e_\theta
    \end{align*}
    is a self-similar solution of
\eqref{SEL} on $\mathbb{R}^2\setminus\{0\}$.\\
Furthermore, any self-similar solution  $(u, d, p): \mathbb{R}^2\setminus\{0\}
\to \mathbb{R}^2\times\mathbb{S}^1\times \mathbb{R}$ of\eqref{SEL} is one of the above forms.
\end{theorem}

\bigskip

For $n\geq3$, we will allow {\it the director field $d$ to take value in
$\mathbb{S}^{n-1}$} in the following discussions.

First, we obtain a rigidity result of  solutions, with scaling-invariant bounds.

\begin{theorem} \label{main_small}
For $n\ge 3$, there exists an $\varepsilon_0>0$ such that if $(u, d)\in C^\infty(\R^n\setminus\{0\}, \R^n\times\mathbb S^{n-1})$ 
is a solution of the system \eqref{SEL}, satisfying
\begin{equation}\label{small1}
\max\big\{|u(x)|, |\nabla d(x)|\big\}\le \frac{\varepsilon_0}{|x|}, \ \forall x\in \R^n\setminus \{0\},
\end{equation}
then i) if $n\ge 4$, then $u\equiv 0$ and $d$ is a constant map; and\\
ii) if $n=3$, then $d$ is a constant map, and $u$ is either $0$ or
a Laudau solution of \eqref{SNS}.
\end{theorem}

When dealing with self-similar solutions,
we can improve Theorm \ref{main_small} and obtain

\begin{theorem} \label{main_sym}
For $n\ge 3$, let $(u, d)\in C^\infty(\R^n\setminus\{0\}, \R^n\times\mathbb S^{n-1})$ be a self-similar solution of
the system \eqref{SEL}.
If, in addition,
\begin{equation}\label{small22}
|\nabla d(x)|<\frac{1}{2|x|},  \ x\in \R^n\setminus\{0\}.
\end{equation}
Then i) if $n\ge 4$, then $u\equiv 0$ and $d$ is a constant map, and\\
ii) if $n=3$, then $u$ is either $0$ or a Laudau solution of \eqref{SNS}, and $d$ is a constant map.
\end{theorem}

\medskip
 When restricted to dimension $n=4$, by adapting the energy method from \cite{{BangGuiLiuWangXie23}}, we can relax the smallness assumption \eqref{small1} and obtain

\begin{theorem} 
\label{4Dthm}
There exists a constant $\varepsilon_1>0$ such that if $(u,d)\in C^\infty(\mathbb{R}^4\setminus \{0\}, \mathbb{R}^4 \times \mathbb{S}^3)$ is a solution to \eqref{SEL} in $\mathbb{R}^4\setminus \{0\}$ satisfying
\begin{equation}
\label{main_cond213}
|u(x)|\le \frac{C}{|x|},\ \  \ |\nabla d(x)|<\frac{\varepsilon_1}{|x|}, \ \  \ \forall x\in\mathbb R^4\setminus\{0\},
\end{equation}
for some constant $C>0$, then $u\equiv 0$ and $d$ is a constant map. 
\end{theorem}

\begin{remark} The constant $\varepsilon_1$ in \eqref{main_cond213} of Theorem \ref{4Dthm} can not be removed, since any nontrivial harmonic map
$\phi:\mathbb S^3\to\mathbb S^3$  induces a nontrivial solution $(0, \phi(\frac{x}{|x|}))$ of \eqref{SEL}. In fact, we can choose 
$$
\varepsilon_1:=
\min\Big\{\big\|\nabla\phi\big\|_{C^0(\mathbb S^3)}\ \big|\ 
\phi\in C^\infty(\mathbb{S}^3,\mathbb{S}^3) 
\mbox{ is a nontrivial harmonic map}\Big\}>0.
$$
\end{remark}

\medskip
This paper is organized as follows. Theorem \ref{2Dthm}, is proved in Section \ref{2d}, while  Theorems \ref{main_small} and \ref{main_sym}, are proved in Sections \ref{sec_small} and \ref{sec_sym} respectively. Theorem \ref{4Dthm} is proved in Section \ref{sec_4d}.

\section{Proof of Theorem \ref{2Dthm}}
\label{2d}
 This section is devoted to the proof of Theorem \ref{2Dthm}. 
 
 Let $(r,\theta)$ denote the polar coordinate in $\mathbb{R}^2$, and $\{e_r,e_\theta\}$ be the standard base of the polar coordinate. 
Let $(u,p,d)$ be a self-similar solution to \eqref{SEL} in $\mathbb{R}^2\setminus \{0\}$, with $d(\mathbb{R}^2)\subset\mathbb{S}^1$. One can decompose the solution $(u,p,d)$
into
    \begin{align*}
        u=\frac{f(\theta) e_r + v (\theta)e_\theta}{r},\quad p=\frac{q(\theta)}{r^2},
        \quad
        d=\cos (\xi(\theta)) e_r + \sin(\xi (\theta)) e_\theta,
    \end{align*}
    where $f,v,q$ are $2\pi$-periodic and
  $\xi(\theta+2\pi)=\xi(\theta)$ (mod $2\pi$).
By direct calculations, the equation \eqref{SEL} becomes
    \begin{align}
    \label{SSeq}
    \left\{
    \begin{aligned}
        -f''+vf'-f^2-v^2-2q &= -(\xi'+1)^2,
        \\
        q'-2f'&=
        \left( 
        (\xi'+1)^2
        \right)',
        \\
        v'=0,
        \ \ \xi''-v(\xi'+1)
        &=0,
    \end{aligned}
    \right.
    \end{align}
where $'$ denotes differentiation with respect to $\theta$.

\begin{proof}[Proof of Theorem \ref{2Dthm}]
To solve \eqref{SSeq},
we divide the discussion into two cases:

\smallskip
\noindent\emph{Case 1) $v=0$.} Then, from the fourth equation in
\eqref{SSeq}, we have that $\xi''=0$ so that
$
        \xi(\theta)= m\theta+ \theta_1
$
for some integer $m\in\mathbb{Z}$ and $\theta_1\in \mathbb{R}$. The second equation in \eqref{SSeq} implies that
$
        q=2f+C_1
$
for some constant $C_1\in \mathbb{R}$. Substituting
these into the first equation of \eqref{SSeq}, one can get
    \begin{align} 
    \label{f_eq}
        -f''-f^2-4f=2C_1-(m+1)^2.
    \end{align}
Note that the constant $2C_1-(m+1)^2$ is not a given data but one of the unknowns.
The equation \eqref{f_eq} is equivalent to the one of $\varphi$ on page 960 in \cite{GuillodWittwer15SIAM} via $\varphi=-f, a=0$. 
Note that \eqref{f_eq} always admits a constant solution $f$
by choosing an appropriate constant $C_1$. 
On other hand, 
there exists a nontrivial solution $f=f_{\Phi,k}$  to \eqref{f_eq}, with minimal period $\frac{2\pi}{k}$, satisfying
    \begin{align*}
        \int_0^{2\pi} f_{\Phi,k} (\theta) \, d\theta =\Phi.
    \end{align*}
   if and only if $4+\frac{\Phi}{\pi}\leq k^2, k\in \mathbb{N}$. 
   See equation (2.10) in pages 960-961 of \cite{GuillodWittwer15SIAM} for details.
The value of $C_1$ can be found by integrating  \eqref{f_eq} over $(0, 2\pi)$:
    \begin{align*}
        C_1= -\frac{1}{4\pi} \int_0^{2\pi} f^2(\theta) \,d\theta
        -\frac{2\Phi}{\pi} +\frac{1}{2}(m+1)^2.
    \end{align*}
Thus we obtain 
    \begin{align*}
        q(\theta) =
        2f(\theta)
        -\frac{1}{4\pi} \int_0^{2\pi} f^2(\tau) d\tau 
        -\frac{2\Phi}{\pi} 
        +\frac{1}{2}(m+1)^2.
    \end{align*}

\medskip
\noindent\emph{Case 2) $v\neq 0$.} In this case, we claim
that $\xi'+1=0$, for, otherwise, it follows from the last two equations of \eqref{SSeq} that
    \begin{align*}
        \xi(\theta)=\frac{c}{v}e^{v\theta}-\theta+b
    \end{align*}
for some constants $c,b\in \mathbb{R}$. This contradicts to $\xi(\theta)-\xi(\theta+2\pi)=0$ (mod $2\pi$). Therefore
$
        \xi(\theta)=-\theta+\theta_1
$
for some constant $\theta_1\in \mathbb{R}$.
The second equation of \eqref{SSeq} implies
    \begin{align}
    \label{peq474}
        q=2f+\frac{C}{2}
    \end{align}
    for some constant $C\in \mathbb{R}$. The
the first equation of \eqref{SSeq}  reduces to
    \begin{align}
    \label{feq479}
        f''-vf'+f^2+4f+v^2+C=0
    \end{align}
for some constant $C\in \mathbb{R}$. 

Multiplying \eqref{feq479} by $f'$ and integrating
over $(0,2\pi)$, and using the $2\pi$-periodicity of $f$, we obtain
$$\int_0^{2\pi} (f'(\theta))^2\,d\theta=0.$$
This yields that $f$ is a constant. Hence \eqref{feq479} implies that $C=- f^2 -4f -v^2$  so that by \eqref{peq474} we have 
that  $q=-\frac{1}{2} (f^2+v^2).$
This finishes the proof.
\end{proof}

\section{Proof for Theorem \ref{main_small}}
\label{sec_small}

This section is devoted to the proof of Theorem 
\ref{main_small}. 
First observe that by the definition
of $L^{n,\infty}$, the condition \eqref{small1} implies that $(u,\nabla d)\in L^{n,\infty}(\R^n)$, and
\begin{align}\label{small2}
\big\|(u,\nabla d)\big\|_{L^{n,\infty}(\R^n)}\le C\varepsilon_0.
\end{align}

For $R>0$, denote by $\chi_{B_R}$ the characteristic function of $B_R$. Let
$$\mathbb P: L^2(\R^n,\R^n)\to L^2_{\rm{div}}(\R^n,\R^n)$$
denote the Leray projection operator, see \cite{Galdi11}. Denote by $G$ the fundamental solution of $-\Delta$ in $\R^n$.
Define $(\widehat{u},\widehat{d}):\R^n\to\R^n\times\R^n$ by letting
\begin{equation*}
\begin{cases}
\displaystyle\widehat{u}(x)=\int_{\R^n} \nabla_y G(x-y) \mathbb P\big((u\otimes u+\nabla d\odot\nabla d)\chi_{B_R}\big)(y) \,dy,\\
\displaystyle\widehat{d}(x)=-\int_{\R^n} G(x-y)\big((u\cdot\nabla d+|\nabla d|^2 d)\chi_{B_R}\big)(y)\,dy.
\end{cases}
\end{equation*}
Then, by the Reisz potential estimates between Lorentz spaces (see \cite{Ziemer1989}), we have that for any $1<p<\frac{n}2$, 
$\big(\ \widehat{u},\nabla\widehat{d}\ \big)\in L^{\frac{np}{n-p}, p}(\R^n)$, along with the following estimates: 
\begin{align}\label{lorentz-est1}
\displaystyle\big\|\widehat{u}\big\|_{L^{\frac{np}{n-p}, p}(\R^n)}
&\leq C\big\|\mathbb P\big((u\otimes u+\nabla d\odot\nabla d)\chi_{B_R}\big)\big\|_{L^{p,p}(\R^n)}\nonumber\\
&\leq C \big\||u|^2+|\nabla d|^2\big\|_{L^{p}(B_R)}\nonumber\\
&\leq C \big\|u\big\|_{L^{n,\infty}(B_R)} \big\|u\big\|_{L^{\frac{np}{n-p},p}(B_R)}+
 \big\|\nabla d\big\|_{L^{n,\infty}(B_R)} \big\|\nabla d\big\|_{L^{\frac{np}{n-p},p}(B_R)}\nonumber\\
&\le C\varepsilon_0\big(\big\|u\big\|_{L^{\frac{np}{n-p},p}(B_R)}+\big\|\nabla d\big\|_{L^{\frac{np}{n-p},p}(B_R)}\big),
\end{align}
and
\begin{align}\label{lorentz-est2}
\displaystyle\big\|\nabla\widehat{d}\big\|_{L^{\frac{np}{n-p}, p}(\R^n)}
&\leq C \big\|(u\cdot\nabla d+|\nabla d|^2 d)\chi_{B_R}\big)\big\|_{L^{p,p}(\R^n)}\nonumber\\
&\leq C \big\||u||\nabla d|+|\nabla d|^2\big\|_{L^{p}(B_R)}\nonumber\\
&\leq C\big(\big\|u\big\|_{L^{n,\infty}(B_R)} +\big\|\nabla d\big\|_{L^{n,\infty}(B_R)}\big)\big\|\nabla d\big\|_{L^{\frac{np}{n-p},p}(B_R)}\nonumber\\
&\le C\varepsilon_0\big\|\nabla d\big\|_{L^{\frac{np}{n-p},p}(B_R)}.
\end{align}
From the definition, we see
that $(\widehat{u}, \widehat{d})$ is a weak solution of  
\begin{equation}\label{SEL2}
\begin{cases}
-\Delta \widehat{u}+\nabla \widehat{p}=-\nabla\cdot(u\otimes u+\nabla d\odot\nabla d),\ \nabla\cdot \widehat{u}=0,\\
\Delta \widehat{d}=-|\nabla d|^2 d-u\cdot\nabla d,
\end{cases}
\ \ {\rm{in}}\ \ B_R.
\end{equation}
If we define $v=u-\widehat{u}$, $q=p-\widehat{p}$,  and $e=d-\widehat{d}$, then $(v, q, e)$ is a weak solution of 
\begin{equation}\label{SEL3}
\begin{cases}
-\Delta v+\nabla q=0,\
\nabla\cdot v=0,\\
\Delta e=0.
\end{cases}
\ \ {\rm{in}}\ \ B_R\setminus\{0\}.
\end{equation}
One can verify that for any $n\ge 3$, $d$ is, in fact, a  weak solution of 
$$
\Delta d+|\nabla d|^2 d=u\cdot\nabla d\ \ {\rm{in}}\ \  \R^n.
$$
Hence for any $n\ge 3$, $e$ solves $\Delta e=0$ weakly in $B_R$. From the regularity theory of the Laplace equation,  
$e\in C^\infty(B_R)$ and, for
any $0<\theta<1$, 
\begin{align}\label{lorentz-est4}
&(\theta R)^{1-\frac{n-p}{p}}\big\|\nabla e\big\|_{L^{\frac{np}{n-p},p}(B_{\theta R})}
\le C\theta R^{1-\frac{n-p}{p}}\big\|\nabla e\big\|_{L^{\frac{np}{n-p},p}(B_{R})}\nonumber\\
&\le C\theta \Big(R^{1-\frac{n-p}{p}}\big\|\nabla d\big\|_{L^{\frac{np}{n-p},p}(B_{R})}
+R^{1-\frac{n-p}{p}}\big\|\nabla\widehat{d}\big\|_{L^{\frac{np}{n-p},p}(B_{R})}\Big)\nonumber\\
&\le C\theta \Big(R^{1-\frac{n-p}{p}}\big\|\nabla d\big\|_{L^{\frac{np}{n-p},p}(B_{R})}
+C\varepsilon_0R^{1-\frac{n-p}{p}}\big\|\nabla d\big\|_{L^{\frac{np}{n-p},p}(B_{R})}\Big)\nonumber\\
&\le C\theta R^{1-\frac{n-p}{p}}\big\|\nabla d\big\|_{L^{\frac{np}{n-p},p}(B_{R})},
\end{align}
where we have used \eqref{lorentz-est2}. Combining \eqref{lorentz-est2} and \eqref{lorentz-est4}, we obtain that 
for $n\ge 3$, and for any $0<\theta<1$, 
\begin{align}\label{lorentz-est5}
&(\theta R)^{1-\frac{n-p}{p}}\big\|\nabla d\big\|_{L^{\frac{np}{n-p},p}(B_{\theta R})}\nonumber\\
&\le C\theta R^{1-\frac{n-p}{p}}\big\|\nabla d\big\|_{L^{\frac{np}{n-p},p}(B_{R})}
+C\varepsilon_0\theta^{2-\frac{n}{p}} R^{1-\frac{n-p}{p}}\big\|\nabla d\big\|_{L^{\frac{np}{n-p},p}(B_{R})}\nonumber\\
&\le \theta^\frac12 R^{1-\frac{n-p}{p}}\big\|\nabla d\big\|_{L^{\frac{np}{n-p},p}(B_{R})},
\end{align}
provided $\theta$, $\varepsilon_0$ are chosen so that $0<\theta<\frac{1}{4C^2}$
and $2C\varepsilon_0\le \theta^{\frac{n}{p}-\frac32}$.

By iterating \eqref{lorentz-est5} $k$-times, we obtain that for any $n\ge 3$ and $\theta\in (0,1)$,
\begin{align}\label{lorentz-est6}
(\theta^k R)^{1-\frac{n-p}{p}}\big\|\nabla d\big\|_{L^{\frac{np}{n-p},p}(B_{\theta^k R})}
\le \theta^\frac{k}2 R^{1-\frac{n-p}{p}}\big\|\nabla d\big\|_{L^{\frac{np}{n-p},p}(B_{R})}.
\end{align}
This implies that for any $n\ge 3$,
\begin{align}\label{lorentz-est7}
r^{1-\frac{n-p}{p}}\big\|\nabla d\big\|_{L^{\frac{np}{n-p},p}(B_{r})}
&\le C\big(\frac{r}{R}\big)^\frac12 R^{1-\frac{n-p}{p}}\big\|\nabla d\big\|_{L^{\frac{np}{n-p},p}(B_{R})}\nonumber\\
&\le C\big(\frac{r}{R}\big)^\frac12\big\|\nabla d\big\|_{L^{n,\infty}(B_R)}\le C\varepsilon_0 \big(\frac{r}{R}\big)^\frac12
\end{align}
holds for for any $0<r\le \frac{R}2$.  Sending $R\to\infty$ in \eqref{lorentz-est7} implies that $\nabla d\equiv 0$ in $\R^n$ so that
$d$ is a constant map in $\R^n$.

From this, we deduce that $u\in C^\infty(\R^n\setminus\{0\})$ solves the Navier-Stokes equation in $\R^n\setminus\{0\}$. Since
$$
|u(x)|\le \frac{\varepsilon_0}{|x|}, \ \forall x\not=0,
$$
$u\equiv 0$ when $n\geq 4$. See \cite{BangGuiLiuWangXie23}. 
When $n=3$, $u$ is either 0 or a Landau solution, provided $\varepsilon_0>0$ is sufficiently small, according to \cite[Corollary 1.5]{MiuraTsai2012}. \qed

\section{Proof of Theorem \ref{main_sym}}
\label{sec_sym}

This section is devoted to the proof of Theorem
\ref{main_sym} on self-similar solutions of 
\eqref{SEL}. Using the spherical coordinates, we can write
$$
(u, p, d)(r,\theta)=\big(\frac{1}{r} u(\theta), \frac{1}{r^2} p(\theta), d(\theta)\big), \ \forall (r,\theta)\in \R_+\times \mathbb S^{n-1}.
$$
Denote by $v(\theta)=u(\theta)-\langle u(\theta), \theta\rangle\theta, \ \theta\in \mathbb S^{n-1},$the tangential component of
$u$ on $\mathbb S^{n-1}$. 
Then $d$ solves the equation of advective harmonic maps 
on $\mathbb S^{n-1}$:
\begin{equation}\label{ahm}
\Delta_{\mathbb S^{n-1}} d+|\nabla_{\mathbb S^{n-1}} d|^2 d=v\cdot\nabla_{\mathbb S^{n-1}} d \ \ {\rm{on}}\ \ \mathbb S^{n-1}.
\end{equation}
Moreover, by \eqref{small22}, $d$ satisfies
\begin{equation}\label{small3}
\big|\nabla_{\mathbb S^{n-1}} d(\theta)\big|<\frac12,\ \ \forall\theta\in\mathbb S^{n-1}. 
\end{equation}
Without loss of generality, we may assume $d(0',1)=(0',1)$ with $0'=(0,\cdots, 0)\in\mathbb R^{n-1}$.
It follows from \eqref{small3} that
\begin{equation}\label{halfsphere}
d(\mathbb S^{n-1})\subset \mathbb S^{n-1}_+=\Big\{x=(x', x_n)\in\mathbb S^{n-1}: \ x_n>0\Big\}.
\end{equation}
Let $\phi(\theta)={\rm{dist}}^2_{\mathbb S^{n-1}}\big(\theta, (0',1)\big)$ be the square of distance function on $\mathbb S^{n-1}$ between $\theta\in\mathbb S^{n-1}$ and the north pole $(0',1)$. It is well-known (see \cite{Jost1984}) that $\phi(\cdot)$ is a strictly convex function on $\mathbb S^{n-1}_+$, and by the chain rule we have that
\begin{align}\label{ahm1}
&\Delta_{\mathbb S^{n-1}} \phi(d(\theta))-v(\theta)\cdot\nabla_{\mathbb S^{n-1}} \phi(d(\theta))\nonumber\\
&={\rm{tr}}\Big(\nabla^2_{\mathbb S^{n-1}}\phi(d(\theta))\big(\nabla_{\mathbb S^{n-1}} d(\theta), \nabla_{\mathbb S^{n-1}} d(\theta)\big)\Big)\ge c_0|\nabla_{\mathbb S^{n-1}} d(\theta)|^2\quad{\rm{on}}\ \ \mathbb S^{n-1},
\end{align}
for a positive constant $c_0>0$. 
It follows from the strong maximum principle that $\phi(d)$ must be constant on $\mathbb S^{n-1}$, which, combined with \eqref{ahm1}, implies
that $\nabla_{\mathbb S^{n-1}} d\equiv 0$ on $\mathbb S^{n-1}$. Hence $d\equiv (0',1)$. 

Now we see that $u$ is a homogeneous ($-1$) solution of the Navier-Stokes equation in $\R^n\setminus\{0\}$. Hence by \cite{Sverak11}
that $u\equiv 0$ in $\R^n$ for $n\ge 4$, and $u$ is either $0$ or a Landau solution in $\R^3$ when $n=3$. \qed

\section{Proof of Theorem \ref{4Dthm}}
\label{sec_4d}
To prove Theorem \ref{4Dthm}, we first need a Lemma on estimates of derivatives of $u,p,d$.
\begin{lemma}
\label{lemma1} For $n\ge 2$,
let $(u,p,d)$ be a smooth solution to \eqref{SEL} in $\mathbb{R}^n\setminus \{0\}$, satisfying
    \begin{align}
    \label{cond510}
        |u(x)|\leq \frac{C_0}{|x|},
        \quad
        |\nabla d(x)|\leq \frac{C_0}{|x|}
        \quad 
        \text{in }\ \mathbb{R}^n\setminus \{0\}
    \end{align}
for some constant $C_0>0$. Then for any non-negative integer $k$, it holds that
    \begin{align}
    \label{main_est}
        |\nabla^k u(x)|\leq \frac{C_k}{|x|^{k+1}},
        \
        |\nabla ^k p(x)| \leq \frac{C_k}{|x|^{k+2}},
        \
        |\nabla^{k+1} d(x)|\leq \frac{C_k}{|x|^{k+1}} 
        \
        \text{in }\ \mathbb{R}^n\setminus \{0\}
    \end{align}
for some constants $C_k>0$.
\end{lemma}
\begin{proof}[Proof of Lemma \ref{lemma1}]
This Lemma was proven by \cite[Lemma X.9.2]{Galdi11} for the steady Navier-Stokes equation.
Here we will follow this proof by making necessary changes.

Fix $x_0 \in \mathbb{R}^n\setminus \{0\}$, and let $R=|x_0|/3$. Define
    \begin{align*}
        \tilde{u}(x)= R \, u(Rx+x_0),
        \quad
        \tilde{p} (x) = R^2 \, p(Rx+x_0),
        \quad
        \tilde{d} (x) = d(Rx+x_0).
    \end{align*}
Then $(\tilde u, \tilde p, \tilde d)$ solves \eqref{SEL} in $B_2(0)$, and \eqref{cond510} implies that 
    \begin{align}
    \label{est542}
        |\tilde{u}(x)|\leq C_0, 
        \quad |\nabla \tilde d(x)|\leq C_0,
        \ \forall\ x\in B_2(0).
    \end{align}
Applying estimates of the Stokes system \cite[Lemma 2.12]{Tsai18} to the first equation of \eqref{SEL},  we obtain that,
for $1<q<\infty$,
    \begin{align}
    \label{est550}
        \|\nabla \tilde{u}\|_{L^q(B_1(0))}\leq C
        \big(\|\tilde{u}\|_{L^q(B_2(0))}
        +\| \tilde{u}\|_{L^{2q} (B_2(0))}^2
        +\|\nabla \tilde{d}\|_{L^{2q}(B_2(0))}^2
        \big)
        \leq C
    \end{align}
for some constant $C$ depending on $n, q$, and $C_0$
in \eqref{est542}.

Applying estimates for the Laplace equation \cite[Theorem 9.11]{GilbargTrudinger98} to third equation of $\eqref{SEL}$
yields
    \begin{align}
    \label{est562}
        \|\nabla^2 \tilde{d}\|_{L^q (B_1(0))} 
        \leq 
        C (1+\|\tilde{u}\|_{L^\infty(B_2(0))})\|\nabla\tilde{d}\|_{L^q(B_2(0))}+C\|\nabla \tilde{d}\|_{L^{2q}(B_2(0))}^2
        \leq C
    \end{align}
for some constant $C$ depending only on $n,q$, and 
$C_0$ in \eqref{est542}.

Applying estimates of the Stokes system again yields that for $1<q<\infty$,
    \begin{align*}
        \|\nabla^2 \tilde{u}\|_{L^q (B_{1/2}(0))} 
        \leq 
        C(
        \|\tilde{u} \|_{W^{1,q} (B_1(0))}
        +
        \||\tilde{u}||\nabla \tilde{u}|\|_{L^q(B_1(0))}
        +
        \||\nabla \tilde d||\nabla^2 \tilde d|\|_{L^q(B_1(0))}
        )
        \leq C,
    \end{align*}
where we use \eqref{est542}, \eqref{est550}, \eqref{est562}. This estimate, together with
the Sobolev embedding, implies that 
\begin{equation*}
\label{grad1}\|\nabla \tilde{u}\|_{L^\infty(B_{\frac12}(0))}\leq C.
\end{equation*}
This implies
    \begin{align*}\label{grad2}
        |\nabla u(x_0)|\leq \frac{C}{R^2} \le \frac{C}{|x_0|^2}.
    \end{align*}

Next, applyig the estimates of Laplace equations, we  obtain that
    \begin{align*}
        &\|\nabla^3 \tilde{d}\|_{L^q (B_\frac34(0))}\\ 
        &\leq 
        C \Big( \|\nabla \tilde {d}\|_{L^q(B_1(0))} + \|(|\tilde {u}|+|\nabla \tilde{d}|)|\nabla^2\tilde{d}|\|_{L^q(B_1(0))}+\||\nabla\tilde{u}||\nabla\tilde{d}|\|_{L^q(B_1(0))}\Big)\leq C.
    \end{align*}
This, combined with  the Sobolev embedding, yields
that
    \begin{align*}
        \big\|\nabla^2 \tilde{d}\big\|_{L^\infty(B_\frac12(0))}\leq C. 
    \end{align*}
This implies 
    \begin{align*}
        |\nabla ^2 d(x_0)|\leq \frac{C}{R^2}\le \frac{C}{|x_0|^2}.
    \end{align*}

Through a boot-strap argument, one can derive the estimates \eqref{main_est} for $u, d$. 
Since $$\nabla p=\Delta u-u\cdot\nabla u-\nabla\cdot(\nabla d\odot\nabla d),$$
the estimates of $\nabla^k p$, for $k\geq 1$, follows from that of $u$ and $d$. 

Finally, we want to derive the estimate of $p$. 
Define $\bar{p}(r)$ by
    \begin{align*}
        \bar{p}(r)=
        \frac{1}{|\partial B_1|} \int_{\partial B_1} p(r,\omega) \,d\sigma(\omega).
    \end{align*}
Then 
    \begin{align*}
        \left| \frac{\partial \bar p}{\partial r} (r) \right|
        =
        \left|
        \frac{1}{|\partial B_1|}
        \int_{\partial B_1} \frac{\partial p}{\partial r} (r,\omega) \, d\sigma
        \right|
        \leq
        \frac{1}{|\partial B_1|}
        \int_{\partial B_1} \frac{C_1}{r^3} \, d\sigma
        \leq \frac{C_1}{r^3}.
    \end{align*}
Therefore, for $0<r_1<r_2,$
    \begin{equation}
    \label{eq636}
        |\bar{p}(r_2)-\bar{p}(r_1)|
        =
        \left|
        \int _{r_1}^{r_2} \frac{\partial \bar{p}}{\partial r} \, dr
        \right|
        \leq
        \int_{r_1}^{r_2} \frac{C_1}{r^3} \, dr
        =C_1 \left(\frac{1}{r_1^2}- \frac{1}{r_2^2} \right).
    \end{equation}
Hence it follows that $\displaystyle\lim _{r\to\infty} \bar{p} (r)$ exists and is finite.
Since $p$ is defined up to a constant, we can assume without loss of generality that $\displaystyle
        \lim_{r\to\infty} \bar{p}(r)=0.
$        
By taking the limit as $r_2 \to\infty$ in \eqref{eq636}, we obtain that for all $r>0$,
    \begin{align*}
        |\bar{p}(r)| \leq \frac{C_1}{r^2}.
    \end{align*}
Observe that for $r>0$ and $\omega\in \mathbb{S}^{n-1}$, 
{
    \begin{align*}
        |p(r, \omega)-\bar{p}(r) |
        &=\left|
        p(r,\omega) 
        -
        \frac{1}{|\partial B_1|} \int _{\partial B_1} p (r,\theta) \, d\sigma
        \right|
        \leq
         \frac{1}{|\partial B_1|}  \int _{\partial B_1}
         r |\nabla p (r,\theta)| \, d\sigma
         \\
         &\leq
         \frac{1}{|\partial B_1|} \int _{\partial B_1}
         \frac{C_1}{r^2} \, d\sigma= \frac{C_1}{r^2}.
    \end{align*}
}
Therefore,
    \begin{align*}
        |p(x)| \leq \frac{2C_1}{|x|^2}, \ \forall x\in\mathbb{R}^n\setminus\{0\}.
    \end{align*}
This finishes the proof.
\end{proof}

Now, we are ready to prove Theorem \ref{4Dthm}.
\begin{proof}[Proof of Theorem \ref{4Dthm}]
Multiplying the third equation of \eqref{SEL} by $\Delta d+|\nabla d|^2 d$ and the first equation of \eqref{SEL} by $u$, and then adding the resulting equations,  we obtain    
    \begin{equation}
    \label{eq528}
     |\nabla u|^2+|\Delta d+ |\nabla d|^2 d|^2
     =\Delta (\frac12|u|^2)-u\cdot\nabla \big(\frac12|\nabla d|^2+\frac12|u|^2+p\big)
\ {\rm{in}}\ \mathbb{R}^n\setminus\{0\}.   \end{equation}
Integrating \eqref{eq528} over the annual region $A_{r,R}=B_R\setminus B_r, 0<r<R$, we obtain
\begin{align}
\label{main_energ}
\begin{aligned}
0\leq &\int_{A_{r,R}}(|\nabla u|^2+|\Delta d+|\nabla d|^2d|^2)\,dx
=\int_{\partial A_{r, R}} \left( \left\langle u, \frac{\partial u}{\partial\nu} \right\rangle - H \langle u,\nu\rangle \right)\,d\sigma
\\
&=\int_{\partial B_R}  
\left(
\left \langle u, \frac{\partial u}{\partial r}\right \rangle 
- H \left \langle u, \frac{x}{|x|} \right \rangle 
\right)\,d\sigma\\
&\quad-\int_{\partial B_r}  
\left( 
\left \langle u, \frac{\partial u}{\partial r} \right\rangle 
- H \left \langle u,\frac{x}{|x|} \right\rangle \right)
\,d\sigma=h(R)-h(r)
\end{aligned}
\end{align}
where $H=\frac12|\nabla d|^2+\frac12|u|^2+p$ denotes the generalized head pressure function, and
$$
h(\tau):=\int_{\partial B_\tau}  
\left(
\left \langle u, \frac{\partial u}{\partial r} \right\rangle 
- H \left \langle u, \frac{x}{|x|} \right\rangle
\right)\,d\sigma, \ \tau>0.
$$
Using \eqref{main_cond213}, one can apply Lemma \ref{lemma1} to obtain
\begin{equation*}
|u(x)|^2+|\nabla u(x)|+|p(x)|+|\nabla d(x)|^2\le \frac{C}{|x|^2}, \ x\in\mathbb R^4\setminus\{0\},
\end{equation*}
so that $h\in L^\infty (0,\infty).$
On the other hand, \eqref{main_energ} implies that $h(R)$ is  increasing with respect to $R>0$. Thus 
both $\displaystyle\lim_{R\to \infty}h(R)$ and $\displaystyle\lim_{r\to 0}h(r)$ exist and are finite. Hence
$$
\int_{\mathbb R^4} (|\nabla u|^2+|\Delta d+|\nabla d|^2d|^2)\,dx<+\infty.
$$
This inequality, combined with the fact that $\displaystyle\lim_{|x|\to\infty} u(x)=0$ and the Sobolev embedding theorem, implies that 
$$\int_{\mathbb{R}^4}|u(x)|^4\,dx<\infty.$$
By Fubini's theorem,  we can find two sequences $r_i\to 0$ and
$R_i\to\infty$ such that
$$
\lim_{i\to\infty}r_i\int_{\partial B_{r_i}} |u(x)|^4\,d\sigma 
=0,\ 
\lim_{i\to\infty}R_i\int_{\partial B_{R_i}} |u(x)|^4\,d\sigma 
=0.
$$
Since
\begin{align*}
|h(R)|
&\le CR^{-2}\int_{\partial B_R}|u(x)|\,d\sigma
\le C\big(R\int_{\partial B_R}|u(x)|^4\,d\sigma\big)^\frac14, 
\end{align*}
we conclude that
$$\displaystyle\lim_{i\to \infty}h(r_i)=0,
\ \ \displaystyle\lim_{i\to \infty}h(R_i)=0,$$
hence
$$
\int_{\mathbb R^4} (|\nabla u|^2+|\Delta d+|\nabla d|^2d|^2)\,dx=0.
$$
Therefore, $u\equiv 0$ in $\mathbb R^4$ and $d:\mathbb R^4\to\mathbb S^3$ is a harmonic map satisfying
$$
|\nabla d(x)|<\frac{\varepsilon_1}{|x|}, \ \forall x\in\mathbb R^4\setminus\{0\}.
$$
The next Lemma shows that $d$ must be constant. 
This finishes the proof.
\end{proof}

\begin{lemma}\label{liouville} For $n\ge 3$, there exists an $\varepsilon_n>0$ such that if $d\in C^\infty(\R^n\setminus\{0\}, \mathbb S^{n-1})$ is a harmonic map
that satisfies
\begin{equation}\label{small1-1}
|\nabla d(x)|\le \frac{\varepsilon_n}{|x|}, \ \forall x\in \R^n\setminus\{0\},
\end{equation}
then $d$ must be a constant map.
\end{lemma}

\begin{remark} i) The same proof implies Lemma 
\ref{liouville} holds when $\mathbb{S}^{n-1}$ is replaced by any compact smooth Riemannian submanifold $(N, h)\subset\mathbb R^L$ without boundary. \\
ii) Lemma 5.2 can be proved by the same argument as in Theorem \ref{main_small} for the case that $u\equiv 0$. Here we will present a different proof that is based on the regularity theory  harmonic maps. See also \cite{Liao1985} for another proof. 
\end{remark}
\begin{proof}[Proof of Lemma \ref{liouville}] First, it follows from $n\ge 3$ and \eqref{small1-1} that for any $R>0$,
\begin{align}\label{smallh1}
R^{2-n}\int_{B_R} |\nabla d|^2\,dx&\le R^{2-n}\int_{B_R} \frac{\varepsilon_n^2}{|x|^2}\,dx
=\frac{\mathcal{H}^{n-1}(\mathbb S^{n-1})}{n-2} \varepsilon_n^2.
\end{align}

Next, we want to derive the following energy monotonicity
inequality: 
\begin{equation}\label{monotonicity_ineq1}
r^{2-n}\int_{B_r(0)}|\nabla d|^2\,dx\le R^{2-n}\int_{B_R(0)}|\nabla d|^2\,dx
\end{equation}
holds for any $0<r\le R<\infty$. 

Since $d$ is a smooth harmonic map on $\R^n\setminus\{0\}$, 
it is well-known that $d$
satisfies the following stationarity identity:
\begin{align}\label{stat}
\int_{\R^n} \Big(|\nabla d|^2{\rm{div}} Y-2\langle \frac{\partial d}{\partial x_i}, \frac{\partial d}{\partial x_j}\rangle 
\frac{\partial Y^i}{\partial x_j}\Big)\,dx=0,
\ \forall Y\in C_0^\infty(\mathbb{R}^n\setminus\{0\}, \mathbb{R}^n).
\end{align}
For any $\delta>0$, let $\eta_\delta\in C^\infty(\R^n)$ be such that 
$$0\le \eta_\delta\le 1, \  \eta_\delta=0 \ {\rm{in}}\  B_\delta,\ 
\eta_\delta=1 \ {\rm{outside}}\ B_{2\delta}, \ {\rm{and}}\ |\nabla \eta_\delta|\le \frac{8}{\delta},$$
and $\phi_\delta\in C_0^\infty(\R)$ be such that
$$0\le \phi_\delta\le 1, \  \phi_\delta(t)=1 \ {\rm{for}} \ |t|\le R(1-\delta),\ 
\phi_\delta(t)=0 \ {\rm{for}}\ |t|\ge R, \ {\rm{and}}\ |\phi_\delta'|\le \frac{8}{R\delta}.$$
Substituting $Y(x)=\eta_\delta(x) \phi_\delta(|x|) x\in C_0^\infty(\R^n\setminus\{0\},\R^n)$ into
\eqref{stat}, we have
\begin{align*}
0&=\int_{\R^3} \Big(|\nabla d|^2{\rm{div}}(\eta_\delta(x) \phi_\delta(|x|) x)-2\langle \frac{\partial d}{\partial x_i}, \frac{\partial d}{\partial x_j}\rangle 
\frac{\partial (\eta_\delta(x)\phi_\delta(|x|) x^i)}{\partial x_j}\Big)\,dx\\
&=(n-2)\int_{\R^3} \eta_\delta(x)\phi_\delta(|x|) |\nabla d|^2\,dx+\int_{\R^n} \eta_\delta(x) \phi_\delta'(|x|) \big(|x||\nabla d|^2-2\langle \frac{\partial d}{\partial x_i}, \frac{\partial d}{\partial x_j}\rangle
\frac{x^ix^j}{|x|}\big)\,dx \\
&\quad+\int_{\R^3}\phi_\delta(|x|) \Big(|\nabla d|^2 x\cdot \nabla\eta_\delta
-2\langle \frac{\partial d}{\partial x_i}, \frac{\partial d}{\partial x_j}\rangle x^i\frac{\partial \eta_\delta}{\partial x_j}\Big)\,dx.
\end{align*}
Observe that as $\delta\to 0$, 
$$\int_{\R^3} \eta_\delta(x)\phi_\delta(|x|) |\nabla d|^2\,dx\to \int_{B_R(0)} |\nabla d|^2\,dx,$$
\begin{align*}
&\int_{\R^3} \eta_\delta(x) \phi_\delta'(|x|) \big(|x||\nabla d|^2-2\langle \frac{\partial d}{\partial x_i}, \frac{\partial d}{\partial x_j}\rangle
\frac{x^ix^j}{|x|}\big)\,dx\\
&\rightarrow -R\int_{\partial B_R(0)}\big(|\nabla d|^2-2|\frac{\partial d}{\partial |x|}|^2\big)\,d\sigma,
\end{align*}
and
\begin{align*}
&\Big|\int_{\R^3}\phi_\delta(|x|) \Big(|\nabla d|^2 x\cdot \nabla\eta_\delta
-2\langle \frac{\partial d}{\partial x_i}, \frac{\partial d}{\partial x_j}\rangle x^i\frac{\partial \eta_\delta}{\partial x_j}\Big)\,dx\Big|\\
&\le C\delta^{-1} \int_{B_{2\delta}(0)} |x||\nabla d|^2\,dx\le C\int_{B_{2\delta}(0)}|\nabla d|^2\,dx\to 0.
\end{align*}
Hence, by sending $\delta\to 0$, we obtain
$$
(2-n)\int_{B_R(0)}|\nabla d|^2\,dx+R\int_{\partial B_R(0)}\big(|\nabla d|^2-2|\frac{\partial d}{\partial |x|}|^2\big)\,d\sigma=0,
$$
which implies
\begin{equation}\label{monotonicity_ineq2}
\frac{d}{dR} \big(R^{2-n}\int_{B_{R}(0)}|\nabla d|^2\,dx\big)=2R^{2-n}\int_{\partial B_R(0)}|\frac{\partial d}{\partial |x|}|^2\,d\sigma\ge 0.
\end{equation}
Integrating \eqref{monotonicity_ineq2} over $R$ yields \eqref{monotonicity_ineq1}.

Now we want to show $d$ is smooth near the origin. To see this, we first claim that
\begin{equation}\label{small4}
R^{2-n} \int_{B_R(x_0)}|\nabla d|^2\,dx\le C\varepsilon_n^2, \ \forall x_0\in \R^n, \ \ R>0.
\end{equation}
From \eqref{monotonicity_ineq1}
and \eqref{smallh1}, \eqref{small4} holds when $x_0=0$. For $x_0\not=0$, we divide $R$ into two cases:\\

\noindent a)  $0<R<|x_0|$. Since $d\in C^\infty(\R^n\setminus \{0\})$, it follows from the energy monotonicity inequality of $d$ on $\R^n\setminus\{0\}$ that
$$
R^{2-n}\int_{B_R(x_0)} |\nabla d|^2\,dx\le |x_0|^{2-n}\int_{B_{|x_0|}(x_0)} |\nabla d|^2\,dx
\le |x_0|^{2-n}\int_{B_{2|x_0|}(0)}|\nabla d|^2\,dx\le C\varepsilon_n^2,
$$
where we have used \eqref{small1-1} in the last step.\\

\noindent b) $R\ge |x_0|$. In this case, we have
$$
R^{2-n}\int_{B_R(x_0)}|\nabla d|^2\,dx\le R^{2-n}\int_{B_{2R}(0)}|\nabla d|^2\,dx\le C\varepsilon_n^2,
$$ 
where we have used \eqref{small1-1} in the last step.

Now we can apply the small
energy regularity theorem of harmonic maps by
Riviere-Struwe \cite{RiviereStruwe2008} (see, also, Evans \cite{Evans1991} and Bethuel \cite{Bethuel1993}) to conclude
that there exists an $\varepsilon_n>0$ such that if \eqref{small1-1} holds, then \eqref{small4} implies that $d\in C^\infty(\R^n,\mathbb S^{n-1})$, and
\begin{equation}\label{gradient_bound}
R^2\sup_{B_{\frac{R}2}(0)}|\nabla d|^2\le C R^{2-n}\int_{B_R(0)}|\nabla d|^2\,dx\le C\varepsilon_n^2,\ \ \forall R>0.
\end{equation}
After sending $R\to\infty$, \eqref{gradient_bound} yields $|\nabla d|\equiv 0$ in $\R^n$ so that $d$ must be  constant.
\end{proof}

\noindent{\bf Acknowledgements}. This work was initiated during the first author was visiting Purdue University. He would like to thank Purdue University for the hospitality. The second author is partially supported by NSF grant 2101224 and a Simons Travel Grant.

\bibliographystyle{amsplain}
\bibliography{mybib.bib}

\providecommand{\bysame}{\leavevmode\hbox to3em{\hrulefill}\thinspace}
\providecommand{\MR}{\relax\ifhmode\unskip\space\fi MR }
\providecommand{\MRhref}[2]{%
  \href{http://www.ams.org/mathscinet-getitem?mr=#1}{#2}
}
\providecommand{\href}[2]{#2}
\begin{thebibliography}{10}

\bibitem{BangGuiLiuWangXie23}
Jeaheang Bang, Changfeng Gui, Hao Liu, Yun Wang, and Chunjing Xie,
  \emph{Rigidity of steady solutions to the navier-stokes equations in high
  dimensions}, arXiv preprint arXiv:2306.05184 (2023).

\bibitem{Bethuel1993}
Fabrice Bethuel, \emph{On the singular set of stationary harmonic maps},
  Manscripta Math \textbf{78} (1993), no.~4, 417--443.

\bibitem{BCL1986}
Ha\"im Brezis, Jean-Michel Coron, and Elliott~H. Lieb, \emph{Harmonic maps with
  defects}, Comm. Math. Phys. \textbf{107} (1986), no.~4, 649--705. \MR{868739}

\bibitem{Ericksen1962}
J.~L. Ericksen, \emph{Hydrostatic theory of liquid crystals}, Arch. Rational
  Mech. Anal. \textbf{9} (1962), 371--378. \MR{137403}

\bibitem{Evans1991}
C.~Larence Evans, \emph{Partial regularity for stationary harmonic maps into
  spheres}, Arch. Rational Mech. Anal \textbf{116} (1991), no.~2, 101--113.

\bibitem{Galdi11}
G.~P. Galdi, \emph{An introduction to the mathematical theory of the
  {N}avier-{S}tokes equations}, second ed., Springer Monographs in Mathematics,
  Springer, New York, 2011, Steady-state problems. \MR{2808162}

\bibitem{GilbargTrudinger98}
David Gilbarg and Neil Trudinger, \emph{Elliptic partial differential equations
  of second order}, Springer, 1998.

\bibitem{GuillodWittwer15SIAM}
Julien Guillod and Peter Wittwer, \emph{Generalized scale-invariant solutions
  to the two-dimensional stationary {N}avier-{S}tokes equations}, SIAM J. Math.
  Anal. \textbf{47} (2015), no.~1, 955--968. \MR{3316196}

\bibitem{Hamel17}
Georg Hamel, \emph{Spiralf{\"o}rmige bewegungen z{\"a}her fl{\"u}ssigkeiten.},
  Jahresbericht der deutschen mathematiker-vereinigung \textbf{25} (1917),
  34--60.

\bibitem{Hong2011}
Min-Chun Hong, \emph{Global existence of solutions of the simplified
  {E}ricksen-{L}eslie system in dimension two}, Calc. Var. Partial Differential
  Equations \textbf{40} (2011), no.~1-2, 15--36. \MR{2745194}

\bibitem{HongXin2012}
Min-Chun Hong and Zhouping Xin, \emph{Global existence of solutions of the
  liquid crystal flow for the {O}seen-{F}rank model in {$\Bbb{R}^2$}}, Adv.
  Math. \textbf{231} (2012), no.~3-4, 1364--1400. \MR{2964608}

\bibitem{HLLW2016}
Tao Huang, Fanghua Lin, Chun Liu, and Changyou Wang, \emph{Finite time
  singularity of the nematic liquid crystal flow in dimension three}, Arch.
  Ration. Mech. Anal. \textbf{221} (2016), no.~3, 1223--1254. \MR{3509000}

\bibitem{Jost1984}
J\"urgen Jost, \emph{Harmonic mappings between {R}iemannian manifolds},
  Proceedings of the Centre for Mathematical Analysis, Australian National
  University, vol.~4, Australian National University, Centre for Mathematical
  Analysis, Canberra, 1984. \MR{756629}

\bibitem{LLWWZ2022}
Chen-Chih Lai, Fanghua Lin, Changyou Wang, Juncheng Wei, and Yifu Zhou,
  \emph{Finite time blowup for the nematic liquid crystal flow in dimension
  two}, Comm. Pure Appl. Math. \textbf{75} (2022), no.~1, 128--196.
  \MR{4373169}

\bibitem{Landau44}
L.~Landau, \emph{A new exact solution of {N}avier-{S}tokes equations}, C. R.
  (Doklady) Acad. Sci. URSS (N.S.) \textbf{43} (1944), 286--288. \MR{0011205}

\bibitem{Leslie1968}
F.~M. Leslie, \emph{Some constitutive equations for liquid crystals}, Arch.
  Rational Mech. Anal. \textbf{28} (1968), no.~4, 265--283. \MR{1553506}

\bibitem{LTX2016}
Jinkai Li, Edriss~S. Titi, and Zhouping Xin, \emph{On the uniqueness of weak
  solutions to the {E}ricksen-{L}eslie liquid crystal model in {$\Bbb{R}^2$}},
  Math. Models Methods Appl. Sci. \textbf{26} (2016), no.~4, 803--822.
  \MR{3460623}

\bibitem{Liao1985}
Guojun Liao, \emph{A regularity theorem for harmonic maps with small energy},
  J. Differential Geom. \textbf{22} (1985), no.~2, 233--241. \MR{834278}

\bibitem{Lin1989}
Fang-Hua Lin, \emph{Nonlinear theory of defects in nematic liquid crystals;
  phase transition and flow phenomena}, Comm. Pure Appl. Math. \textbf{42}
  (1989), no.~6, 789--814. \MR{1003435}

\bibitem{LLW2010}
Fanghua Lin, Junyu Lin, and Changyou Wang, \emph{Liquid crystal flows in two
  dimensions}, Arch. Ration. Mech. Anal. \textbf{197} (2010), no.~1, 297--336.
  \MR{2646822}

\bibitem{LW2011}
Fanghua Lin and Changyou Wang, \emph{On the uniqueness of heat flow of harmonic
  maps and hydrodynamic flow of nematic liquid crystals}, Chinese Ann. Math.
  Ser. B \textbf{31} (2010), no.~6, 921--938. \MR{2745211}

\bibitem{LinWang2016}
\bysame, \emph{Global existence of weak solutions of the nematic liquid crystal
  flow in dimension three}, Comm. Pure Appl. Math. \textbf{69} (2016), no.~8,
  1532--1571. \MR{3518239}

\bibitem{MiuraTsai2012}
Hideyuki Miura and Tai-Peng Tsai, \emph{Point singularities of 3{D} stationary
  {N}avier-{S}tokes flows}, J. Math. Fluid Mech. \textbf{14} (2012), no.~1,
  33--41. \MR{2891188}

\bibitem{RiviereStruwe2008}
Tristan Rivi\`ere and Michael Struwe, \emph{Partial regularity for harmonic
  maps and related problems}, Comm. Pure Appl. Math. \textbf{61} (2008), no.~4,
  451--463. \MR{2383929}

\bibitem{TianXin98}
Gang Tian and Zhouping Xin, \emph{One-point singular solutions to the
  {N}avier-{S}tokes equations}, Topol. Methods Nonlinear Anal. \textbf{11}
  (1998), no.~1, 135--145. \MR{1642049}

\bibitem{Tsai18}
Tai-Peng Tsai, \emph{Lectures on {N}avier-{S}tokes equations}, Graduate Studies
  in Mathematics, vol. 192, American Mathematical Society, Providence, RI,
  2018. \MR{3822765}

\bibitem{Sverak11}
Vladim\'{\i}r \v{S}ver\'{a}k, \emph{On {L}andau's solutions of the
  {N}avier-{S}tokes equations}, vol. 179, 2011, Problems in mathematical
  analysis. No. 61, pp.~208--228. \MR{3014106}

\bibitem{Wang2011}
Changyou Wang, \emph{Well-posedness for the heat flow of harmonic maps and the
  liquid crystal flow with rough initial data}, Arch. Ration. Mech. Anal.
  \textbf{200} (2011), no.~1, 1--19. \MR{2781584}

\bibitem{XuZhang2012}
Xiang Xu and Zhifei Zhang, \emph{Global regularity and uniqueness of weak
  solution for the 2-{D} liquid crystal flows}, J. Differential Equations
  \textbf{252} (2012), no.~2, 1169--1181. \MR{2853534}

\bibitem{Ziemer1989}
William~P. Ziemer, \emph{Weakly differentiable functions}, Graduate Texts in
  Mathematics, vol. 120, Springer-Verlag, New York, 1989, Sobolev spaces and
  functions of bounded variation. \MR{1014685}

\end{thebibliography}
\end{document}